\documentclass[12pt]{article}

\usepackage[paper=letterpaper,margin=0.7in,twoside=false,includehead]{geometry}

\usepackage{amsfonts,amsmath,amsthm, amssymb, thmtools, thm-restate} 
 
\usepackage{hyperref}
\usepackage[capitalize]{cleveref}

\usepackage{graphicx}
\usepackage{fancyhdr}
\usepackage{xparse}
\usepackage{mathtools}

\usepackage{verbatim}

\usepackage[shortlabels]{enumitem}
\usepackage[normalem]{ulem}

\SetEnumitemKey{thmparts}{topsep=2pt plus 1pt minus1pt, itemsep=1pt plus0.5pt minus0.5pt}

\usepackage[T1]{fontenc}

\usepackage{tikz}
\usepackage{tkz-berge}
\usetikzlibrary{decorations.pathreplacing}

\newcommand{\N}{\mathbb{N}}

\newcommand{\symd}{\bigtriangleup}
\DeclarePairedDelimiter\ceil{\lceil}{\rceil}
\DeclarePairedDelimiter\floor{\lfloor}{\rfloor}

\DeclarePairedDelimiter\abs{|}{|}
\DeclarePairedDelimiter\parens{(}{)}
\DeclarePairedDelimiter\set{\{}{\}}

\DeclarePairedDelimiterX\setof[2]{\{}{\}}{#1\,:\,#2}

\DeclareDocumentCommand{\k}{O{t}}{k^{#1}}

\DeclareMathOperator{\Ts}{T}
\DeclareMathOperator{\CTs}{CT}

\DeclareDocumentCommand{\G}{O{r}O{n}}{G^*_{#2,#1}}
\DeclareDocumentCommand{\Gs}{O{r}O{n}O{\ell}}{G^{#3*}_{#2,#1}}
\DeclareDocumentCommand{\T}{O{n}O{r}}{\Ts_{#2}(#1)}
\DeclareDocumentCommand{\CT}{O{m}O{r}}{\CTs_{#2}(#1)}

\newtheorem{thm}{Theorem}[section]
\newtheorem{lem}[thm]{Lemma}
\newtheorem{cor}[thm]{Corollary}

\newtheorem{prop}[thm]{Proposition}

\theoremstyle{definition}
\newtheorem{defn}[thm]{Definition}
\newtheorem{qu}{Question}

\newtheorem{obs}[thm]{Observation}
\newtheorem{case}{Case}
\numberwithin{case}{thm}

\numberwithin{subcase}{case}

\newcommand{\A}{\mathcal{A}}

\DeclareMathOperator{\Ham}{Ham}
\DeclareMathOperator{\ex}{ex}

\DeclareDocumentCommand{\Gell}{O{r}O{n}O{\ell}}{\mathcal{G}_{#2,#1}^{#3}}
\DeclareDocumentCommand{\Hl}{O{r}O{n}O{\ell}}{\mathcal{H}_{#2,#1}^{#3}}
\DeclareDocumentCommand{\Hlhat}{O{r}O{n}O{\ell}}{\widehat{\mathcal{H}}_{#2,#1}^{#3}}

\begin{document}

\pagestyle{plain}

\title{Ore plus Tur\'{a}n}
\author{A. Dawkins and R. Kirsch}
\date{October 1, 2026}
\maketitle
\abstract{
Ore in 1961 determined the maximum number of edges in graphs not containing a Hamiltonian cycle, and Tur\'{a}n in 1941 found the maximum number of edges in graphs not containing a $K_{r+1}$. Motivated by the work of Adamus in 2009 and Ferrero and Lesniak in 2018 on the maximum number of edges in $r$-partite non-Hamiltonian graphs, we find the maximum number of edges in $K_{r+1}$-free non-Hamiltonian graphs. We determine the extremal graphs, which are different for $n \ge 26$ with $r \ge 3$, for small values of $n$ with $r \in \set{4,5}$, and for all values of $n$ with $r=2$. In the first case we show that the same graphs achieve the maximum number of $t$-cliques for every $t \ge 2$. Both the edge and clique results are proved via extremal theorems for graphs satisfying a condition on low-degree vertices related to P\'{o}sa's theorem, which we apply more broadly in a second paper. Both extend from Hamiltonicity to chorded pancyclicity.
}

\section{Introduction}\label{sec:intro}

One of the foundational results in extremal graph theory is Tur\'{a}n's theorem, which gives the maximum number of edges in $n$-vertex, $K_{r+1}$-free graphs. We write $e(G)$ for the number of edges in $G$. 
\begin{thm}[Tur\'an \cite{turan}]\label{thm:turan}
    For every $n \ge 1$ and $r \ge 1$, if $G$ is an $n$-vertex, $K_{r+1}$-free graph, then $e(G) \le e(\T)$, where $\T$ is the $n$-vertex complete $r$-partite graph with parts as balanced as possible, and equality holds if and only if $G \cong \T$. 
\end{thm} The problem of determining the extremal number $\ex(n,F)$, the maximum number of edges in an $n$-vertex, $F$-free graph (where $F$ is a graph not depending on $n$), has been studied extensively \cite{Simonovits1997}. Separately, extensive research has been done on minimum degree, $\sigma_2$ (also called Ore-type), and edge density conditions sufficient to guarantee the existence of Hamiltonian cycles and related properties. See for example \cite{Araujo} for a brief history. As opposed to the usual phrasing of Ore's theorem, we choose to think of these edge density conditions from an extremal perspective and phrase the theorems in the form of a maximum number of edges.

\begin{thm}[Ore \cite{OreEdgeCond}]\label{thm:ore} If $G$ is an $n$-vertex, non-Hamiltonian graph, then $e(G) \le e(K_{n-1}) +1$. For $n \ge 2$, equality holds if and only if $G$ is a $K_{n-1}$ plus a pendant edge, unless $G \cong K_{3,1,1}$.
\end{thm}

When considering only bipartite graphs, significantly weaker conditions suffice to guarantee Hamiltonicity and related properties \cite{AdamusBalanced, AdamusUnbalanced, BaggaVarma, MitchemSchmeichel, MoonMoser}. Recent work has considered $n$-vertex, $r$-partite graphs for $2 < r < n$ \cite{Adamus, Araujo, CFGJL, ChenJacobson,DKPT19,DS21,FerreroLesniak}. For example, Adamus determined the maximum number of edges among $n$-vertex balanced tripartite graphs that are not Hamiltonian. In order to state this result, we first define a family of graphs.
\begin{defn}
	For any integers $n \ge 2$ and $r \ge 2$, let $\G$ be the $n$-vertex graph formed from $\T[n-1]$ by adding a new vertex $x$ and an edge $\set{x,y}$, where $y$ is a vertex in a largest part of $\T[n-1]$.
\end{defn}

The graph $\G$ has a vertex $x$ of degree one so is not Hamiltonian.

\begin{thm}[Adamus \cite{Adamus}]\label{thm:adamus}
If $G$ is an $n$-vertex, balanced tripartite graph that is not Hamiltonian, then $e(G) \le e(\G[3])$. Equality holds if $G \cong \G[3]$.
\end{thm}

Ferrero and Lesniak considered non-Hamiltonian $r$-partite graphs that are not necessarily balanced. They determined the maximum number of edges for such graphs with given part sizes. 
The following corollary of their work generalizes \Cref{thm:adamus}.

\begin{thm}[Corollary of Ferrero and Lesniak \cite{FerreroLesniak}]\label{cor:FL}
For sufficiently large $n$ and $r \ge 3$, if $G$ is an $n$-vertex, $r$-partite graph that is not Hamiltonian, then $e(G) \le e(\G)$. Equality holds if $G \cong \G$.
\end{thm}

\begin{figure}[h!]\begin{center}
		\begin{tabular}{ccc}
			& Unrestricted & Non-Hamiltonian\\
			Unrestricted & \tikzstyle{vx}=[inner sep=1.5pt,circle,fill=black,draw=black]
			\tikzstyle{edge}=[thick]
			\begin{tikzpicture}[scale=.6]
				\draw[fill=gray!20] (2,2) circle [radius=1.8];
			\end{tikzpicture} & \tikzstyle{vx}=[inner sep=1.5pt,circle,fill=black,draw=black]
			\tikzstyle{edge}=[thick]
			\begin{tikzpicture}[scale=.6]
                \draw[fill=gray!20] (2,2) circle [radius=1.8];
				\foreach \x/\y [count=\i] in {
                3.35/2.65, 4.5/3}
				{
					\node[vx] (V\i) at (\x,\y) {};
				}
				
				\draw[edge] (V1)--(V2);
			\end{tikzpicture}\\
			&&	Ore \cite{OreEdgeCond}\\
			$K_{r+1}$-free & \tikzstyle{vx}=[inner sep=1.5pt,circle,fill=black,draw=black]
			\tikzstyle{edge}=[thick]
			\begin{tikzpicture}[scale=.6]

             \draw[gray!20, line width=8pt] (3,3) -- (1,3);
                    \draw[gray!20, line width=8pt] (3,3) -- (3,1);
                    \draw[gray!20, line width=8pt] (3,3) -- (1,1);
                    \draw[gray!20, line width=8pt] (1,3) -- (3,1);
                    \draw[gray!20, line width=8pt] (1,3) -- (1,1);
                    \draw[gray!20, line width=8pt] (3,1) -- (1,1);
                    
				\draw[rotate around={-45:(1,3)}, fill=white] (1,3) ellipse [x radius=.5, y radius=.9];
				\draw[rotate around={45:(1,1)}, fill=white] (1,1) ellipse [x radius=.5, y radius=.9];
				\draw[rotate around={45:(3,3)}, fill=white] (3,3) ellipse [x radius=.5, y radius=.9];
				\draw[rotate around={-45:(3,1)}, fill=white] (3,1) ellipse [x radius=.5, y radius=.9];
                
			\end{tikzpicture} & \tikzstyle{vx}=[inner sep=1.5pt,circle,fill=black,draw=black]
			\tikzstyle{edge}=[thick]
			\begin{tikzpicture}[scale=.6]

            \draw[gray!20, line width=8pt] (1,1) -- (3,1);
                    \draw[gray!20, line width=8pt] (3,3) -- (1,3);
                    \draw[gray!20, line width=8pt] (3,3) -- (3,1);
                    \draw[gray!20, line width=8pt] (3,3) -- (1,1);
                    \draw[gray!20, line width=8pt] (1,3) -- (3,1);
                    \draw[gray!20, line width=8pt] (1,3) -- (1,1);

				\draw[rotate around={-45:(1,3)}, fill=white] (1,3) ellipse [x radius=.5, y radius=.9];
				\draw[rotate around={45:(1,1)}, fill=white] (1,1) ellipse [x radius=.5, y radius=.9];
				\draw[rotate around={45:(3,3)}, fill=white] (3,3) ellipse [x radius=.5, y radius=.9];
				\draw[rotate around={-45:(3,1)}, fill=white] (3,1) ellipse [x radius=.5, y radius=.9];
				\foreach \x/\y [count=\i] in {
                3.35/2.65, 4.5/3}
				{
					\node[vx] (V\i) at (\x,\y) {};
				}
				
				\draw[edge] (V1)--(V2);
			\end{tikzpicture}\\
			& Tur\'an \cite{turan} & $\G$
            
		\end{tabular}
\caption{Extremal graphs for four problems. Gray circles represent complete graphs, and gray bands indicate complete multipartite graphs.\label{fig:tofl}}
\end{center}
\end{figure}

\cref{fig:tofl} shows the extremal graphs for \cref{thm:turan}, \cref{thm:ore}, and \cref{cor:FL}. 
The extremal graph $\G$ in \cref{cor:FL} is a modification of a Tur\'{a}n graph, and it is a modification of the extremal graph of Ore's theorem (\cref{thm:ore}) in which the complete graph is replaced by a Tur\'{a}n graph, as happens in many extremal problems where a $K_{r+1}$ is forbidden (e.g. \cite{Frohmader, KR20, K25, MNNRW23, turan}).

Therefore in this paper we relax the $r$-partite condition and focus on extremal problems for graphs that do not contain a $K_{r+1}$. Our first main theorem, \cref{theorem:kr+1all}, substantially strengthens \cref{cor:FL} by showing that the same conclusion follows from this weaker hypothesis. That is, the graphs $\G$ maximize the number of edges among $n$-vertex, $K_{r+1}$-free, non-Hamiltonian graphs for sufficiently large $n$. 
Moreover, we characterize the extremal graphs.

We prove \cref{theorem:kr+1all} using \cref{thm:degcondsummary}, an edge extremal theorem for graphs satisfying a condition on low-degree vertices related to P\'{o}sa's theorem. In a second paper \cite{DK26Stability} we use \cref{thm:degcondsummary} to generalize from Hamiltonicity to a broad range of properties.

Next, we extend this result from maximizing the number of edges to maximizing the number of $t$-cliques. Zykov \cite{Zykov} in 1949 proved that the Tur\'{a}n graph maximizes the number of $t$-cliques (sets of $t$ vertices that induce a $K_t$) among $n$-vertex, $K_{r+1}$-free graphs. This generalizes Tur\'{a}n's theorem to counting cliques rather than edges. More recently, the generalized Tur\'{a}n problem of determining $\ex(n,H,F)$, the maximum number of copies of $H$ in $n$-vertex, $F$-free graphs, has flourished, especially in the case $H=K_t$. Gerbner and Palmer have recently written an interesting survey covering a wide range of generalized Tur\'{a}n problems, discussing methods and results \cite{GP25}. In 2024 Chakraborti and Chen \cite{ChakrabortiChen} proved, as a consequence of a more general theorem on $\ex(n,K_t,C_{\ge k})$, that the extremal graphs of \cref{thm:ore} also maximize the number of $t$-cliques in $n$-vertex non-Hamiltonian graphs.

\cref{cor:cliques} strengthens \cref{theorem:kr+1all}. When we forbid both $K_{r+1}$ and a Hamiltonian cycle, the graphs $\G$ maximize not only the number of edges but also the number of $t$-cliques. Thus, \cref{fig:tofl} shows the extremal graphs for $t$-cliques as well as edges. \cref{cor:posaclique} strengthens both \cref{thm:degcondsummary} and \cref{cor:cliques}, maximizing the number of $t$-cliques in $K_{r+1}$-free graphs satisfying the same condition on low-degree vertices.

To prove \cref{cor:cliques} we notice that $\G$ is an example of a colex Tur\'{a}n graph, a graph consisting of the first $m$ edges in colexicographic order from an infinite complete $r$-partite graph. We introduce a new application, which may be of independent interest, of a theorem of Frohmader \cite{Frohmader} that relies on the rainbow Kruskal-Katona theorem of Frankl, F\"{u}redi, and Kalai \cite{FFK88} to give the maximum number of $t$-cliques in $m$-edge, $K_{r+1}$-free graphs.

The paper is organized as follows. Preliminary definitions and theorems are stated in \cref{sec:prelim}. In \cref{sec:graphs} we define two families of $n$-vertex, $K_{r+1}$-free graphs, which we later use in characterizing the extremal graphs, and determine which of these graphs are Hamiltonian. Then in \cref{sec:bounds} we show that a certain degree condition implies the maximum numbers of edges for $r \ge 3$ and narrow down the extremal graphs to two families. \cref{sec:edge} pieces together the degree conditions from \cref{sec:prelim} and the results of Sections \ref{sec:graphs} and \ref{sec:bounds} to determine the maximum numbers of edges and the extremal graphs that attain them for $r \ge 3$. In \cref{sec:clique} we introduce the application of Frohmader's theorem and maximize the number of $t$-cliques for all $t \ge 2$. We extend these results for Hamiltonicity to chorded pancyclicity in \cref{sec:chordedpancyclicity}. In \cref{sec:r2} we show that in the $r=2$ case, the extremal graphs are nearly balanced complete bipartite graphs. We suggest several directions for future research in \cref{sec:open}.

\section{Preliminaries}\label{sec:prelim}

In this paper, graphs are simple graphs, having no loops or multiple edges. If $H$ is a subgraph of $G$, we write $\partial_G H$, or $\partial H$ where the graph $G$ is clear, for the set of edges of $G$ having exactly one vertex in $V(H)$. For a vertex $v$ in $G$ and a subgraph $H$ of $G$ we write $d_H(v)$ for the number of neighbors of $v$ in $H$.

In addition to Ore's theorem on the number of edges (\cref{thm:ore}), we use the following two sufficient conditions for Hamiltonicity based on the degree sequence and minimum degree.

\begin{thm}[Chv\'atal \cite{Chvatal}]\label{thm:kstable}
    If $G$ is a non-Hamiltonian graph with degrees $d_1 \le \cdots \le d_n$, then there is an integer $1 \le i \le (n-1)/2$ for which $d_i \le i$ and $d_{n-i} \le n-i-1$.
\end{thm}

We refer to the condition that there is an integer $1 \le i \le (n-\ell-1)/2$ for which $d_i \le i + \ell$ as a \emph{P\'{o}sa-type} condition on low-degree vertices, after \cite{Posa62}, and the condition that there is an integer $1 \le i \le (n-1-\ell)/2$ for which $d_i \le i+\ell$ and $d_{n-i-\ell} \le n-i-1$ as a \emph{Chv\'{a}tal-type} condition, after \cite{Chvatal}; a generalization of \cref{thm:kstable} using $\ell$ appears in \cite{BC76}.

\begin{thm}[Dirac \cite{Dirac}]\label{thm:stablemindeg}
    If an $n$-vertex graph $G$ has $\delta(G) \ge n/2$, then $G$ is Hamiltonian.
\end{thm}

\cref{thm:stablemindeg} can be used to characterize which complete multipartite graphs are Hamiltonian.

\begin{prop}\label{prop:partitekHam}
    Let $G$ be a complete multipartite graph on $n$ vertices whose largest part has size $m$.  
    Then $G$ is Hamiltonian if and only if $m \le n/2$.
\end{prop}

\begin{proof}
   Suppose $G$ has a part of size greater than $n/2$. Then the other parts have at most $(n-1)/2$ vertices. As any cycle in $G$ is a cyclic ordering of vertices of $G$ in which the vertices in the largest part are non-consecutive, $G$ is not Hamiltonian.

   Conversely, suppose the largest part of $G$ has size $m \le n/2$. The degree of a vertex in a part of size $p$ is $n-p$, so the minimum degree of $G$ is $\delta(G)=n-m\ge n/2$. This minimum degree condition, by \cref{thm:stablemindeg}, implies that $G$ is Hamiltonian.
\end{proof}

We frequently use the following fact to compare expressions for numbers of edges.
\begin{prop}\label{floorprop}
    For any real numbers $x$ and $y$, $\floor{x} - \floor{y} > x - 1-y$.
\end{prop}
\begin{proof}
    The definition of the floor function implies that $x-1< \floor{x}$ and $\floor{y} \le y$.
\end{proof}

We use the following bounds and exact counts for the numbers of edges in Tur\'{a}n graphs.

\begin{prop}[\cite{Vatter,Lavrov}]\label{prop:turanedgecount}
    Let $n \ge r \ge 1$ and $s = n \mod r$. Then
    \[
        e(\T) = \frac{r-1}{2r}n^2 - \frac{s(r-s)}{2r},\quad\quad\quad
       \frac{r-1}{2r}n^2 - \frac{r}{8} \le e(\T) \le \frac{r-1}{2r}n^2,
    \]
    and, for every $1 \le r \le 7$, $e(\T) = \floor[\big]{\frac{r-1}{2r}n^2}$.
\end{prop}

\section{Extremal Graphs}\label{sec:graphs}

In this section we define two families of potential extremal graphs, $\Gell$ and $\Hl$, and address the Hamiltonicity of these families of graphs in \cref{prop:gell is extremal} and \cref{Hfam}.

For $-1 \le \ell \le \ceil{(r-1)(n-1)/r}-1$ we write $\Gs[r][n][\ell]$ for the colex Tur\'{a}n graph $\CT$ (see \cref{sec:clique}), where $m = e(\T[n-1][r]) + \ell+1$. The graph $\Gs[r][n]$ is formed from $\T[n-1]$ by adding a new vertex $x$ such that the neighborhood of $x$ is isomorphic to $\T[\ell+1][r-1]$ and disjoint from a smallest part of the $\T[n-1]$. Notice that $\G = \Gs[r][n][0]$.

\begin{defn}[$\Gell$ and $\Hl$]\label{def:graphs}
Let $\Gell$ be the set of $n$-vertex, $K_{r+1}$-free graphs consisting of $\T[n-1][r]$ plus a vertex of degree $\ell+1$. Notice that $\G$ is one of the two graphs in $\Gell[r][n][0]$, and $\Gs \in \Gell$ provided that $\ell+1 \le \ceil{(r-1)(n-1)/r}$, so $\Gell$ is nonempty if $n \ge (1+1/(r-1))(\ell+1)+1$.

Let $\Hl$ be the set of $n$-vertex, $K_{r+1}$-free graphs consisting of $\T[(n+1+\ell)/2][r]$ plus an independent set of $(n-1-\ell)/2$ vertices of degree $(n-1+\ell)/2$. Notice that $\Hl$ is empty when $n \equiv \ell \pmod{2}$.
\end{defn}

In Theorems \ref{thm:degcondsummary} and \ref{theorem:kr+1all} and Corollaries \ref{cor:cliques} and \ref{cor:posaclique} we prove that the graphs in $\Gell$ are extremal for sufficiently large $n$. First we confirm here that these graphs satisfy the requirements on low-degree vertices and, in the case $\ell=0$, are not Hamiltonian.

\begin{prop}\label{prop:gell is extremal}
    For integers $2 \le r \le n-1$ and $-1 \le \ell \le n-3$, suppose $G \in \Gell$. Then $G$ is an $n$-vertex, $K_{r+1}$-free graph. Let $d_1 \le \cdots \le d_n$ be its degrees. Then $d_j \le j+\ell$ for $j=1$, and $e(G)=e(\T[n-1])+\ell+1$. If $\ell=0$, then $G$ is not Hamiltonian and $e(G)=e(\T[n-1])+1$.
\end{prop}

\begin{proof}
    By the definition of $\Gell$, $G$ contains a vertex of degree $\ell+1$, so satisfies $d_j \le j+\ell$ for $j=1$. Since $G$ consists of a $\T[n-1][r]$ plus a vertex of degree $\ell+1$, we have $e(G)=e(\T[n-1])+\ell+1$. If $G \in \Gell[r][n][0]$, $G$ contains a vertex $v$ of degree $1$ so has no Hamiltonian cycle.
\end{proof}

We now turn our attention to properties of graphs in the family $\Hl$. For $4 \le r \le 7$, we prove in \cref{thm:kr+1general} that all extremal graphs are in either $\Gell$ or $\Hl$.

\begin{lem}\label{Hfamsize}
    Let $2 \le r \le n-1$ and $-1 \le \ell \le n-3$. If $G \in \Hl$, then $n \le 4r-\ell-3$ and $G$ consists of 
    an independent set $V(J)$ of $(n-1-\ell)/2$ vertices of degree $(n-1+\ell)/2$ and $H \cong \T[(n+1+\ell)/2][r]$ containing at least one partite set of size 1 and all other partite sets of size at most 2. 
    Furthermore, all vertices in partite sets of size $2$ are adjacent to all vertices of $J$.
\end{lem}

\begin{proof}
    Let $G \in \Hl$. Then $G$ is an $n$-vertex, $K_{r+1}$-free graph consisting of an independent set $V(J)$ of $(n-1-\ell)/2$ vertices of degree $(n-1+\ell)/2$ in $G$, and $H \cong \T[(n+1+\ell)/2][r]$. Each vertex $v$ in $J$ has degree $(n-1+\ell)/2$ in $H$ so has exactly one non-neighbor in $H$. If the non-neighbor of $v$ were in a partite set of $H$ of size greater than $1$, then $H$ would have exactly $r$ partite sets each of size at least $1$, and $v$ would be adjacent to at least one vertex in each of the $r$ partite sets, forming a $K_{r+1}$ in $N[v]$, so the non-neighbor of $v$ is in a part of size $1$, and all vertices in partite sets of size greater than 1 must be adjacent to all vertices of $J$. Thus $H$ has some partite set of size $1$, and as $H$ is a Tur\'{a}n graph, all of its partite sets have size at most $2$. Therefore, $|V(H)| = \frac{n+1+\ell}{2} \le 2r-1$, so $n \le 4r-\ell-3$.
\end{proof}

\begin{lem}\label{Hfamell}
Let $G \in \Hl$ where $\ell \ge -1$ and $n \ge \ell+5$. Then either
\begin{enumerate}[(a)]
    \item $G$ is a complete multipartite graph with one partite set of size $(n+1-\ell)/2$ and the other partite sets of sizes at most 2, or
    \item\label{part:deg} for every integer $i$ in $1 \le i \le (n-1-\ell)/2$, either $d_i > i+\ell$ or $d_{n-i-\ell} > n-i-1$.
\end{enumerate}
\end{lem}

\begin{proof}
    Let $G$ be such an $n$-vertex, $K_{r+1}$-free graph with degrees $d_1 \le \cdots \le d_n$ consisting of an independent set $V(J)$ of $(n-1-\ell)/2$ vertices of degree $(n-1+\ell)/2$ in $G$, and $H \cong \T[(n+1+\ell)/2][r]$. We consider the equality of the neighborhoods $N(u)$ and $N(v)$ for $u,v \in V(J)$, as a partite set of size $(n+1-\ell)/2$ is only possible when $N(u)=N(v)$ for all $u,v \in V(J)$.
    
\begin{case} $N(u)=N(v)$ for all $u,v \in V(J)$

There is one vertex $h \in V(H)$ that is non-adjacent to all the vertices of $J$, so $G$ is a complete multipartite graph with a part $V(J)\cup \{h\}$ of size $(n+1-\ell)/2$. By \cref{Hfamsize}, the other parts of $G$ have sizes at most $2$ as they are also parts of $H$. 
\end{case}

\begin{case}\label{case:unequal} $N(u)\neq N(v)$ for some $u,v \in V(J)$

In this case, each vertex in $H$ is adjacent to at least one vertex in $J$. By \cref{Hfamsize}, the vertices of $H$ belong to a partite set of size 1 or a partite set of size 2, and the vertices in $J$ are adjacent to all vertices in partite sets of size 2. For every $h \in V(H)$ in a partite set of size 1, $d_H(h) = (n-1+\ell)/2$ and $d_J(h)\ge 1$. For every $h \in V(H)$ in a partite set of size 2, $d_H(h) = (n-3+\ell)/2$ and $d_J(h) = (n-1-\ell)/2$. Using $n \ge \ell+5$, for all $h \in V(H)$, $d_G(h) \ge (n+1+\ell)/2$. Therefore $G$ has degree sequence $((n-1+\ell)/2,\dots, (n-1+\ell)/2, d_{(n+1-\ell)/2}, \dots, d_n)$ where $d_{(n+1-\ell)/2} \ge (n+1+\ell)/2$. For all $1 \le i < (n-1-\ell)/2$ we have $d_i = (n-1+\ell)/2 > i+\ell$, and for $i=(n-1-\ell)/2$ we have $d_{n-i-\ell}=d_{(n+1-\ell))/2}
\ge (n+1+\ell)/2 > n-i-1$. \qedhere
\end{case}
\end{proof}

\begin{cor}\label{Hfam}
Let $n \ge 5$ and $G \in \Hl[r][n][0]$. If $G$ is not Hamiltonian then $G$ is a complete multipartite graph with one partite set of size $(n+1)/2$ and the other partite sets of sizes at most 2.
\end{cor}

\begin{proof}
    By \cref{thm:kstable}, we have the negation of condition \ref{part:deg} in \cref{Hfamell} with $\ell=0$.
\end{proof}

\section{Edge Density Conditions from Degree Conditions}\label{sec:bounds}

In this section we prove an upper bound on the number of edges in an $n$-vertex, $K_{r+1}$-free graph that satisfies a condition on low-degree vertices related to P\'{o}sa's theorem, for all $r \ge 3$. 
Specifically, we prove the following.

\begin{thm}\label{thm:degcondsummary}
    Let $\ell \ge -1$, $r \ge 3$, and \[n \ge \begin{cases} 6\ell+26 & \text{if } r=3\\ \max\{3+ \ell + \frac{4(\ell+2)}{r-3},5+\ell+\frac{r+2\ell+7}{2r-2}\} & \text{if } 4 \le r \le 7\\ \ell+5+\frac{4(\ell+4)}{r-4} & \text{if } r \ge 8\end{cases}\] be integers, or $(\ell,r,n)=(1,8,10)$. Let $G$ be an $n$-vertex, $K_{r+1}$-free graph with degrees $d_1 \le \cdots \le d_n$. If there is an integer $j$ in $1 \le j \le (n-1-\ell)/2$ such that $d_j \le j+\ell$, then \[ e(G) \leq e(\T[n-1][r]) + (\ell+1). \]
    For $r=3$ and $r \ge 8$, equality holds if and only if $G \in \Gell[r][n][\ell]$. For $4 \le r \le 7$, if $G \in \Gell[r][n]$ then equality holds, and if equality holds then either $G \in \Gell[r][n]$ or $G \in \Hl[r][n]$. In all cases, $\Gell$ is nonempty, and the bound is tight.
\end{thm}

Notice that every graph in $\Gell$ has $d_1 \le 1+\ell$, so satisfies the degree condition. 
We give separate proofs for $r \ge 8$, $4 \le r \le 7$, and $r=3$ in the following subsections. 
In each case we see that, when $n$ is large enough compared to $\ell$, $\Gell$ is nonempty (see \cref{def:graphs}), so the upper bound on the number of edges is tight.

\subsection{$r \ge 8$}

In this subsection we determine the maximum number of edges in $n$-vertex, $K_{r+1}$-free graphs satisfying the P\'{o}sa-type condition for every $r \ge 8$.

\begin{thm}\label{thm:kr+1general8}
Let $\ell \ge -1$, $r \ge 8$, and $n\ge \ell+5 + \frac{4(\ell+4)}{r-4}$ be integers, or $(\ell,r,n) = (1,8,10)$. Let $G$ be an $n$-vertex, $K_{r+1}$-free graph with degrees $d_1 \le \cdots \le d_n$. If there is an integer $j$ in $1 \le j \le (n-1-\ell)/2$ such that $d_j \le j+\ell$, then \[ e(G) \leq e(\T[n-1][r]) + (\ell+1). \]
The set $\Gell$ is nonempty, and equality holds if and only if $G \in \Gell$.
\end{thm}

\begin{proof}
The fact that $\Gell$ is nonempty follows from $n\ge \ell+5 + \frac{4(\ell+4)}{r-4} \ge (1+1/(r-1))(\ell+1)+1$ for all $r \ge 8$ and \cref{def:graphs}. 
With $\ell$, $r$, $n$, and $G$ as in the theorem, suppose that for some $1 \le j \le (n-1-\ell)/2$ we have $d_j \le j+\ell$, so $G$ contains a set of $j$ vertices of degree at most $j+\ell$. Let $J$ be the subgraph of $G$ induced by such a set of $j$ vertices. Let $H := G-V(J)$. Notice 
\begin{equation}\label{eq:ej8}
j^2 + \ell j \geq \sum_{v \in V(J)} d_G(v) = \sum_{v \in V(J)} d_J(v) + \sum_{v \in V(J)} d_H(v) = 2e(J)+|\partial J| \geq e(J)+|\partial J|.\end{equation}
As $H$ is $K_{r+1}$-free, by Tur\'{a}n's theorem we have 
\begin{equation}
    e(G) = e(J) + \abs{\partial J} + e(H) \le  j^2+\ell j + e(\T[n-j][r])\label{eq:Hbound8},
\end{equation}
and equality holds if and only if
$e(J)=0$ by \cref{eq:ej8}, every vertex $v \in V(J)$ has $d(v) = j+\ell$ by \cref{eq:ej8}, and $H\cong \T[n-j][r]$ by \cref{eq:Hbound8} and \cref{thm:turan}. 
In the special case $(\ell,r,n)=(1,8,10)$, we have $j \in \set{1,2,3,4}$, and \cref{eq:Hbound8} shows that $e(G) < e(\T[n-1][r])+(\ell+1)$ except when $j=1$.

\begin{case}$j=1$

Substituting $j=1$ into \cref{eq:Hbound8}, we have \begin{align*}
    e(G) &\le e(\T[n-1][r])+(\ell+1).
\end{align*} If equality holds, then $H \cong \T[n-1][r]$ and the other vertex of $G$ has degree $j+\ell$, so $G \in \Gell$.

Conversely, if $G \in \Gell$, then, by \cref{prop:gell is extremal}, $G$ satisfies all hypotheses of the theorem and $e(G) = e(\T[n-1][r])+(\ell+1)$.
\end{case}

\begin{case} $2 \le j \le (n-1-\ell)/2$

Let $s = (n-1)\mod r$. 
First we prove
\begin{equation}\label{eq:123}
\frac{r-1}{2r}\left(2n-2\right)-\frac{j}{2}+\frac{s}{r}\\
    \ge j+1+\ell
\end{equation}
as follows. 
Notice 
\begin{align}
    \frac{2(r-1)n-2(2r-1)-2r\ell+2s}{3r}-j &\ge \frac{2(r-1)n-2(2r-1)-2r\ell+2s}{3r} - \left(\frac{n-1-\ell}{2}\right) \nonumber\\
    &= \frac{(r-4)n-r(\ell+5)+4(s+1)}{6r} \nonumber\\
    &\ge \frac{(r-4)n-r(\ell+5) + 4}{6r}\quad\text{as }s\ge 0\label{eq:s0}\\
    &= \frac{r-4}{6r}\left(n-\frac{r(\ell+5)-4}{r-4}\right)\ge 0.\nonumber
\end{align}
We obtain \cref{eq:123} by multiplying both sides of 
$(2(r-1)n-2(2r-1)-2r\ell+2s)/(3r) - j \ge 0$ 
by $3/2$, then adding $j+1+\ell$ to both sides. Note that $s-j+1$ is $(n-j) \mod r$ if $s \ge j-1$ and is negative otherwise. Letting $f(x) = x(r-x)$, in both cases we have $f(s-j+1) \le f((n-j)\mod r)$ since $f(x) < 0$ for $x < 0$ and $f(x) \ge 0$ for $0 \le x \le r-1$. Now, by \cref{prop:turanedgecount},  
\begin{align}
    &e(\T[n-1][r])-e(\T[n-j][r])\nonumber\\
    &= \frac{r-1}{2r}(n-1)^2 - \frac{s(r-s)}{2r} - \left(\frac{r-1}{2r}(n-j)^2 - \frac{((n-j)\mod r)(r-((n-j)\mod r))}{2r}\right)\nonumber\\
    &\ge \frac{r-1}{2r}(n-1)^2 - \frac{s(r-s)}{2r} - \left(\frac{r-1}{2r}(n-j)^2 - \frac{(s-j+1)(r-(s-j+1))}{2r}\right)\label{eq:sineq}\\
    &=(j-1)\left(\frac{r-1}{2r}\left(2n-2\right)-\frac{j}{2}+\frac{s}{r}\right)\nonumber\\
    &\ge (j-1)(j+1+\ell)\quad \text{ by \cref{eq:123}}\nonumber \\
    &= j^2+\ell j -(\ell+1),\nonumber 
\end{align} 
which then implies $e(G) \le e(\T[n-1][r])+(\ell+1)$ by \cref{eq:Hbound8}.  
The inequality is strict when $s\ne 0$ by \cref{eq:s0} and strict when $s=0$ by \cref{eq:sineq} (using the fact that $j > 1$, so $s-j+1 < 0$ and $f(s-j+1) < 0 \le f((n-j)\pmod{r})$).\qedhere
\end{case}
\end{proof}

\subsection{$4 \le r \le 7$}
In this subsection we prove an upper bound on the number of edges in $n$-vertex, $K_{r+1}$-free graphs satisfying the P\'{o}sa-type condition for $4 \le r \le 7$, where by \cref{prop:turanedgecount} we have $e(\T[n][r]) = \floor{\frac{r-1}{2r}n^2}$ for every $n \ge r$.

\begin{thm}\label{thm:kr+1general}
Let $\ell \ge -1$ be an integer, $4 \le r \le 7$, and $n \geq \max\{3+ \ell + \frac{4(\ell+2)}{r-3},5+\ell+\frac{r+2\ell+7}{2r-2}\}$. Let $G$ be an $n$-vertex, $K_{r+1}$-free graph with degrees $d_1 \le \cdots \le d_n$. If there is an integer $j$ in $1 \le j \le (n-1-\ell)/2$ such that $d_j \le j+\ell$, then 
\begin{align*}
    e(G) &\le e(\T[n-1][r]) + (\ell+1).
\end{align*}
The set $\Gell$ is nonempty, and if $G \in \Gell$ then equality holds. Furthermore, if $e(G)=e(\T[n-1][r]) + (\ell+1)$, then either $G \in \Gell$ (so $G$ satisfies $d_j \le j + \ell$ for $j=1$), or $G \in \Hl$ (so $G$ satisfies $d_j \le j + \ell$ for $j=(n-1-\ell)/2$). 
\end{thm}

\begin{proof}
    The fact that $\Gell$ is nonempty follows from $n \geq 3+ \ell + \frac{4(\ell+2)}{r-3} \ge (1+1/(r-1))(\ell+1)+1$ for all $4 \le r \le 7$ and \cref{def:graphs}. 
    Let $\ell \ge -1$ be an integer and $G$ be an $n$-vertex, $K_{r+1}$-free graph. Suppose for some $1 \le j \le (n-1-\ell)/2$ we have $d_j \le j+\ell$, so $G$ contains a set of $j$ vertices of degree at most $j+\ell$. Let $J$ be the subgraph of $G$ induced by such a set of $j$ vertices. Notice 
\begin{equation}\label{eq:ej}
j^2 + \ell j \geq \sum_{v \in V(J)} d(v) = 2e(J)+|\partial J| \geq e(J)+|\partial J|.\end{equation}
Let $H := G-V(J)$. As $H$ is $K_{r+1}$-free, by Tur\'{a}n's theorem and \cref{prop:turanedgecount} we have 
\begin{align}
    e(G) &= e(J) + \abs{\partial J} + e(H)\le  j^2+\ell j + \Big\lfloor{\frac{r-1}{2r}(n-j)^2}\Big\rfloor\label{eq:2}\\
        &= \Big\lfloor{\frac{r-1}{2r}n^2-\Big(\frac{r-1}{r}n-\ell\Big)j+\frac{3r-1}{2r}j^2}\Big\rfloor,\label{eq:4}
\end{align}
and if equality holds then
$e(J)=0$ by \cref{eq:ej}, every vertex $v \in V(J)$ has $d(v) = j+\ell$ by \cref{eq:ej}, and $H\cong \T[n-j][r]$ by \cref{eq:2} and \cref{thm:turan}. 
We show that the bound given by \cref{eq:4} is at most $e(\T[n-1][r]) + (\ell+1)$ for $j=1$ and $j= (n-1-\ell)/2$. Using \cref{eq:4}, for $j=1$ we have 
\begin{align*}
& \frac{r-1}{2r}n^2 -\Big(\frac{r-1}{r}\Big)n+\frac{2\ell r +3r-1}{2r} - \left(\frac{r-1}{2r}n^2-\Big(\frac{r-1}{r}n-\ell\Big)j+\frac{3r-1}{2r}j^2\right)= 0. 
\end{align*} Using \cref{eq:4} for $j=(n-1-\ell)/2$ we have 
\begin{align*}
& \frac{r-1}{2r}n^2 -\Big(\frac{r-1}{r}\Big)n+\frac{2\ell r +3r-1}{2r} - \left(\frac{r-1}{2r}n^2-\Big(\frac{r-1}{r}n-\ell\Big)j+\frac{3r-1}{2r}j^2\right)\\
&= \left(\frac{r-1}{r}n-\ell\right)(j-1) - \frac{3r-1}{2r}(j^2-1)= \frac{j-1}{2r} \left( 2(r-1)n-
2\ell r - (3r-1)(j+1)\right)\\
&= \frac{j-1}{2r} \left( 2(r-1)n-
2\ell r - (3r-1)\left(\frac{n+1-\ell}{2}\right)\right)\\
&= \frac{(j-1)(r-3)}{4r} \left( n- \frac{\ell(r+1)}{r-3} - \frac{3r-1}{r-3} \right)= \frac{(j-1)(r-3)}{4r} \left( n- 3 -\ell- \frac{4(\ell+2)}{r-3} \right) \ge 0,
\end{align*} where the last inequality follows from $r\ge 4$ and $n \ge 3+\ell+\frac{4(\ell+2)}{r-3}$.
Therefore
\begin{align*}
    e(G) &\le \floor[\Big]{\frac{r-1}{2r}n^2-\Big(\frac{r-1}{r}n-\ell\Big)j+\frac{3r-1}{2r}j^2}\le \floor[\Big]{\frac{r-1}{2r}n^2 -\Big(\frac{r-1}{r}\Big)n+\frac{2\ell r +3r-1}{2r}} \\
    &= e(\T[n-1][r]) + (\ell+1)
\end{align*} 
for $j=1$ and $j=(n-1-\ell)/2$. In the case that $j=1$, $H \cong \T[n-1][r]$, and the other vertex of $G$ has degree $j+\ell$, so $G \in \Gell[r][n][\ell]$. 

Conversely, if $G \in \Gell[r][n][\ell]$, then by \cref{prop:gell is extremal} $G$ satisfies all hypotheses of the theorem and $e(G) = e(\T[n-1][r]) + (\ell+1)$.

We now consider separately when $j=2$ and $2<j<(n-1-\ell)/2$ to show that the bound given by \cref{eq:2} is strictly less than $e(\T[n-1][r])+(\ell+1)$ for $1<j<(n-1-\ell)/2$, so if equality holds then $j=1$ and $G \in \Gell$ or $j=(n-1-\ell)/2$ and $G \in \Hl$. 

\begin{case} $j=2$

Using \cref{eq:4}, when $j=2$ we have
\begin{align*}
&e(\T[n-1][r])+(\ell+1)-e(G) \\
&>\frac{r-1}{2r}n^2 - \frac{r-1}{r}n+\frac{2\ell r +3r-1}{2r}-1 -\left(\frac{r-1}{2r}n^2-\frac{2(r-1)}{r}n+2\ell+\frac{4(3r-1)}{2r}\right)\\
&= \frac{r-1}{r}n-\frac{4(3r-1)+2\ell r-r+1}{2r} 
= \frac{r-1}{r}\left(n-\frac{11r+2\ell r-3}{2r-2}\right)\ge 0,
\end{align*}
where the last inequality follows from $r \ge 4$ and $n \ge 5+\ell+\frac{r+2\ell+7}{2r-2} = \frac{11r+2\ell r-3}{2r-2}$.

\end{case}

\begin{case} $2<j<(n-1-\ell)/2$

Since $r \ge 4$, when $2 < j < (n-1-\ell)/2$ we have $j-1 \ge 2$ and $\frac{(3r-1)(j-1)}{4r} \ge \frac{2(3r-1)}{4r} = 1+\frac{2r-2}{4r}>1.$ For $r \ge 4$ and $n \geq 3+ \ell + \frac{4(\ell+2)}{r-3}$ we have 
\begin{align}
    0 &\le n(r-3)-(3+\ell)(r-3)-4(\ell+2)= (4r-4)n-2(2\ell r+3r-1)-(n-1-\ell)(3r-1), \text{so}\nonumber\\
    0 &\le \frac{2r-2}{3r-1}n-\frac{2\ell r}{3r-1}-1-\frac{n-1-\ell}{2}=\frac{2r-2}{3r-1}n-\frac{2\ell r}{3r-1}-1-\frac{n-2-\ell}{2}-\frac{1}{2}.\label{eq:nbound}
\end{align}
Using \cref{eq:4} we have 
\begin{align*}
    &e(\T[n-1][r])+(\ell+1)-e(G) \\
    &> \frac{r-1}{2r}n^2 -\frac{r-1}{r}n+\frac{2\ell r +3r-1}{2r}-\parens[\bigg]{\frac{r-1}{2r}n^2-\Big(\frac{r-1}{r}n-\ell\Big)j+\frac{3r-1}{2r}j^2 +1}\\
    &>\frac{r-1}{2r}n^2 -\frac{r-1}{r}n+\frac{2\ell r +3r-1}{2r}-\parens[\bigg]{\frac{r-1}{2r}n^2-\Big(\frac{r-1}{r}n-\ell\Big)j+\frac{3r-1}{2r}j^2 +\frac{(3r-1)(j-1)}{4r}}\\
    &= \Big(\frac{r-1}{r}n-\ell\Big)(j-1)- \left(\frac{3r-1}{2r}(j^2-1) + \frac{(3r-1)(j-1)}{4r}\right)\\
    &= \frac{(j-1)(3r-1)}{2r}\left(\frac{2r-2}{3r-1}n-\frac{2\ell r}{3r-1}-j-1-\frac{1}{2}\right)\\
    &\ge \frac{(j-1)(3r-1)}{2r}\left(\frac{2r-2}{3r-1}n-\frac{2\ell r}{3r-1}-1-\frac{n-2-\ell}{2}-\frac{1}{2}\right)\ge 0,
\end{align*}
where the last inequality follows from \cref{eq:nbound}. \qedhere
\end{case}
\end{proof}

\subsection{$r=3$}

Now we address the most difficult case, $r=3$. Here the following additional lemma is needed. We use it only when $r=3$ but state it for general $r$.

\begin{lem} \label{lem:k3Jbound}
Let $j \ge 0$ be an integer and $r \ge 1$. Let $G$ be a $K_{r+1}$-free graph on $n$ vertices containing at least $j$ vertices of degree at most $d$. Let $J$ be the subgraph of $G$ induced by such a set of $j$ vertices, and let $H:=G-V(J)$. Let $i$ be the maximum integer such that $H$ contains $i$ disjoint copies of $K_r$. Then $e(J) + \abs{\partial J} \le \floor{j(n-j-i+d)/2}$.
\end{lem}

\begin{proof}
    Let $K$ be the induced subgraph in $H$ on the vertices of $i$ disjoint copies of $K_r$. Let $H' := H-V(K)$, so $\abs{V(H')}= \abs{V(H)} - \abs{V(K)} = n-j-ir$. As $G$ is $K_{r+1}$-free, each $v \in V(J)$ has at most $r-1$ neighbors in each copy of $K_r$ in $K$. By assumption, for every $v \in V(J)$, we have $d_J(v)+d_H(v) = d(v) \le d$, and $d_H(v) = d_{H'}(v) + d_{K}(v) \le (n-j-ir)+i(r-1) = n-j-i$. Therefore 
    \begin{align*}
        e(J)+\abs{\partial J} &= \frac{1}{2}\sum_{v \in V(J)} d_J(v) + \sum_{v \in V(J)} d_H(v)
        =\sum_{v \in V(J)} \left(\frac{1}{2}(d_J(v)+d_H(v))+\frac{1}{2}d_H(v)\right)\\
        &\le\floor[\Big]{\sum_{v \in V(J)} \frac{1}{2}(d)+\frac{1}{2}(n-j-i)}
        =\floor[\Big]{\frac{j}{2}(n-j-i+d)}.\qedhere
    \end{align*}
\end{proof}

To prove our edge density results when $r=3$, we first prove \cref{thm:k4general} on degree conditions that imply an upper bound on the number of edges. 
We split the proof into two cases, addressed in Theorems \ref{thm:k4general1} and \ref{thm:k4general2}, depending on how $j$ compares to $\frac{n-2}{2}-\frac{3\ell}{4}$. 

\begin{thm}\label{thm:k4general1}
Let $\ell \ge -1$ be an integer and $n \geq 6\ell+26$. Let $G$ be an $n$-vertex, $K_{4}$-free graph with degrees $d_1 \le \cdots \le d_n$. If there is an integer $j$ in $1 \le j \le \frac{n-2}{2}-\frac{3\ell}{4}$ such that $d_j \le j+\ell$, then \[ e(G) \leq e(\T[n-1][3]) + (\ell+1)  = \floor[\Big]{\frac{n^2}{3} - \frac{2n}{3} +\frac{3\ell+4}{3}}. \] Equality holds if and only if $G \in \Gell[3][n]$. 
\end{thm}

\begin{proof} Let $\ell \ge -1$ be an integer and $G$ be an $n$-vertex, $K_{4}$-free graph with degrees $d_1 \le \cdots \le d_n$. Suppose there is an integer $j$ in $1 \le j \le \frac{n-2}{2}-\frac{3\ell}{4}$ such that $d_j \le j+\ell$. 
    Let $J$ be the subgraph of $G$ induced by a set of $j$ vertices of degree at most $j+\ell$. 
    Notice \begin{equation}\label{eq:degsumbound} j^2 + \ell j \geq \sum_{v \in V(J)} d(v) = 2e(J)+|\partial J| \geq e(J)+|\partial J|.\end{equation} Let $H := G-V(J)$. Applying Tur\'{a}n's theorem to $H$, we have 
\begin{equation}
    e(G) = e(J) + \abs{\partial J} + e(H)
        \le  j^2+\ell j + \floor[\Big]{\frac{1}{3}(n-j)^2}
        = \floor[\Big]{\frac{n^2}{3}-\left(\frac{2n}{3}-\ell\right)j+\frac{4}{3}j^2}\label{eq:10}, 
\end{equation}
with equality if and only if $e(J)=0$ (by \cref{eq:degsumbound}), every vertex of $J$ has degree $j + \ell$ (by \cref{eq:degsumbound}), and $H \cong \T[n-j][3]$ (by \cref{thm:turan}).

If $G \in \Gell[3]$ then by \cref{prop:gell is extremal} we have $e(G)=e(\T[n-1][3]) + (\ell+1)$ and $d_1 \le \ell+1$. We consider values of $j$ separately to show that the bound given by \cref{eq:10} is strictly less than $e(\T[n-1][3]) + (\ell+1)$ for all $1<j\le \frac{n-2}{2}-\frac{3\ell}{4}$, so if equality holds then $j=1$, in which case we prove that $G \in \Gell[3]$.

\begin{case}$j=1$

When $j=1$, using \cref{eq:10} we have
    $e(G) \le 
    \floor{\frac{n^2}{3}-\frac{2n}{3}+\frac{3\ell+4}{3}} = e(\T[n-1][3]) + (\ell+1)$. If equality holds, then, by the remarks following \cref{eq:10}, $G$ consists of a subgraph $J$ which is one vertex of degree $\ell+1$ and a subgraph $H \cong \T[n-1]$, and by hypothesis $G$ is $K_4$-free, so $G \in \Gell[3]$.
\end{case}

\begin{case}
    $2 \le j < \frac{n-2}{2}-\frac{3\ell}{4}$

For $2 \le j \le 3$, as $n \ge 6\ell+26$ and $\ell \ge -1$, 
        $\displaystyle\frac{2n}{3} \ge \frac{52+12\ell}{3}
        >\frac{42+3\ell}{3}
        > \frac{4j}{3}+\ell+\frac{4}{3}+\frac{1}{j-1}$.

    For $j \ge 4$, as $j<\frac{n-2}{2}-\frac{3\ell}{4}$ is an integer, we have 
    \begin{align*}
        \frac{2n}{3} &\ge \frac{2}{3}\parens[\bigg]{2j+\frac{3\ell}{2}+\frac{5}{2}}
        \ge \frac{2}{3}\parens[\bigg]{2j+\frac{3\ell}{2}+2+\frac{3}{2(j-1)}} = \frac{4j}{3}+\ell+\frac{4}{3}+\frac{1}{j-1}\quad\text{using }j\ge 4.
    \end{align*} 
    
    In both cases, using \cref{eq:10}, we have 
\begin{align*}
    e(\T[n-1][3]) + (\ell+1) - e(G) 
    &>\frac{n^2}{3}-\frac{2n}{3}+\frac{3\ell+4}{3}-1 - \left(\frac{n^2}{3}-\Big(\frac{2n}{3}-\ell\Big)j+\frac{4}{3}j^2\right)\nonumber\\
    &=(j-1)\left(\frac{2n}{3}-\ell-\frac{4j}{3}-\frac{4}{3}-\frac{1}{j-1}\right) >0.
    \end{align*}

\end{case}

\begin{case}$j=\frac{n-2}{2}-\frac{3\ell}{4}$

Substituting $j=\frac{n-2}{2}-\frac{3\ell}{4}$ into \cref{eq:10} yields $e(G) \le \floor{\frac{n^2}{3}-\frac{2n}{3}+\frac{3\ell+4}{3}}=e(\T[n-1][3]) + (\ell+1)$
with equality holding if and only if $G$ consists of $H \cong \T[\frac{n+2}{2}+\frac{3\ell}{4}][3]$ and an independent set $V(J)$ of $\frac{n-2}{2}-\frac{3\ell}{4}$ vertices of degree $j+\ell=\frac{n-2}{2}+\frac{\ell}{4}$. Each vertex in $J$ is non-adjacent to $|V(H)|-(j+\ell)=2+\frac{\ell}{2}$ vertices of $H$. If a vertex in J has a neighbor in every partite set of $H$ then $G$ contains $K_4$ as a subgraph, but $G$ is $K_4$-free. Thus each vertex in $J$ is non-adjacent to some partite set of $H$, so $H$ contains a partite set of size at most $2+\frac{\ell}{2}$, which implies
    \begin{align*}
        \frac{n+2}{2}+\frac{3\ell}{4} &= |V(H)| \le 3\Big(3+\frac{\ell}{2}\Big) -1,\quad\text{so }
    n \le 14+\frac{3\ell}{2}.
    \end{align*}
    For every $\ell \ge -1$, we have the bounds $n \ge 6\ell+26 > 14+\frac{3\ell}{2} \ge n$, a contradiction. 
    Therefore equality does not hold for $j=\frac{n-2}{2}-\frac{3\ell}{4}$.\qedhere
\end{case}
\end{proof}

\begin{thm}\label{thm:k4general2}
Let $\ell \ge -1$ be an integer and $n \geq 6\ell+26$. Let $G$ be an $n$-vertex, $K_{4}$-free graph with degrees $d_1 \le \cdots \le d_n$. If there is an integer $j$ in $\frac{n-2}{2}-\frac{3\ell}{4}+\frac{1}{4} \le j \le \frac{n-1-\ell}{2}$ such that $d_j \le j+\ell$, then \[ e(G) < e(\T[n-1][3]) + (\ell+1)  = \floor[\Big]{\frac{n^2}{3} - \frac{2n}{3} +\frac{3\ell+4}{3}}. \] 
\end{thm}

\begin{proof}
Let $\ell \ge -1$ be an integer and $G$ be an $n$-vertex, $K_{4}$-free graph with degrees $d_1 \le \cdots \le d_n$. Suppose there is an integer $j$ in $\frac{n-2}{2}-\frac{3\ell}{4}+\frac{1}{4} \le j \le (n-1-\ell)/2$ such that $d_j \le j+\ell$. Let $J$ be the subgraph of $G$ induced by a set of $j$ vertices of degree at most $j+\ell$. 

Let $H := G-V(J)$. Let $i$ be the maximum integer such that $H$ contains $i$ disjoint copies of $K_3$. We consider separately when $i=0$ and $i \ge 1$.

\begin{case}$i=0$

When $i=0$, $H$ is $K_3$-free, so applying Tur\'{a}n's theorem to $H$, we have \begin{align*}
    e(G) & = e(J) + \abs{\partial J} + e(H) \le \parens[\Big]{\sum_{v\in V(J)}d(v)} + e(H)
     \le \floor[\Big]{j^2+\ell j+\frac{(n-j)^2}{4}}\\
    &= \floor[\Big]{\frac{5}{4}j^2+\left(\ell -\frac{n}{2}\right)j+\frac{n^2}{4}}
     \le \floor[\Big]{\frac{5n^2}{16}-\frac{3n}{8}+\frac{\ell n}{8}+\frac{\ell}{8}-\frac{3\ell^2}{16}+\frac{5}{16}},
\end{align*}
where the last inequality follows from maximizing the concave-up quadratic function of $j$ by substituting $j= (n-1-\ell)/2$ since the interval $\frac{n-2}{2}-\frac{3\ell}{4}+\frac{1}{4} \le j \le (n-1-\ell)/2$ lies to the right of the vertex of the parabola at $j=\frac{n-2\ell}{5}$. 
When $\ell = -1$, as $n \ge 6\ell+26$, we have 
\begin{align*}
    e(G) \le \floor[\Big]{\frac{5n^2}{16}-\frac{n}{2}} \le \floor[\Big]{\frac{n^2}{3}-\frac{2n}{3}+\frac{3\ell+4}{3}}-1 < e(\T[n-1][3]) + (\ell+1).
\end{align*}
For $\ell \ge 0$, we have
\begin{align*}
    e(\T[n-1][3]) + (\ell+1) -e(G) 
    &>\frac{n^2}{3}-\frac{2n}{3}+\frac{3\ell+4}{3}-1-\left(\frac{5n^2}{16}-\frac{3n}{8}+\frac{\ell n}{8}+\frac{\ell}{8}-\frac{3\ell^2}{16}+\frac{5}{16}\right)\\
    &=\frac{1}{48}\left(n^2-14n-6\ell n+1+42\ell+9\ell^2\right)> 0,
\end{align*}
where the last inequality follows from the fact that $\ell \ge 0$ so $n \ge 6\ell+26 > 14+6\ell$ and $1+42\ell+9\ell^2>0$.
\end{case}

\begin{case} $i\ge 1$

For this case, we prove and then use the inequality
\begin{equation}\label{eq:54l}
    36i^2-24i\ell-24i+54\ell+13+9\ell^2 > 0.
\end{equation}
Viewing the left side of \cref{eq:54l} as a concave-up quadratic function of $i$, we see that it is minimized when $i=(\ell+1)/3$, so
\begin{align*}
    36i^2-24i\ell-24i+54\ell+13+9\ell^2 &\ge 36\parens[\bigg]{\frac{\ell+1}{3}}^2-24\parens[\bigg]{\frac{\ell+1}{3}}\ell-24\parens[\bigg]{\frac{\ell+1}{3}}+54\ell+13+9\ell^2\\
    &= (5\ell+1)(\ell+9) > 0
\end{align*}
for all $\ell \ge 0$. For $\ell = -1$, using the fact that $i \ge 1$,
\[
    36i^2-24i\ell-24i+54\ell+13+9\ell^2 = 36i^2-54+13+9 \ge 4 > 0,
\]
completing the proof of \cref{eq:54l}.

As $i \ge 1$, let $K$ be the induced subgraph of $H$ on the vertices of $i$ disjoint copies of $K_3$. Let $H' := H-V(K)$, so $H'$ is $K_3$-free. Since $G$ is $K_4$-free, each vertex of $H$ has at most two neighbors in each of the $i$ disjoint triangles in $K$, and $\abs{\partial_H K} \le 2i(n-j-3i)$.

As $i$ is the maximum number of disjoint copies of $K_3$ in $H$, we have $i \le (n-j)/3 \le (2n+3+3\ell)/12$, where the latter inequality follows from $j \ge \frac{n-2}{2}-\frac{3\ell}{4}+\frac{1}{4}$. For $n \ge 6\ell+26$, we have \begin{equation}\label{iupperbound}
    2i-\ell \le (2n+3-3\ell)/6 <\frac{n-2}{2}-\frac{3\ell}{4}+\frac{1}{4}.
\end{equation}

Applying \cref{lem:k3Jbound} with $d=j+\ell$ and Tur\'{a}n's theorem to $K$ and $H'$, we have \begin{align*}
    e(G)  &= e(J) +|\partial J|+e(K)+\abs{\partial_H K}+e(H')\\
    &\le \floor[\Big]{\frac{j}{2}\left(n-i+\ell\right)+e(\T[3i][3])+2i(n-j-3i)+e(\T[n-j-3i][2]}\\
    &= \floor[\Big]{\frac{n^2}{4}+\frac{in}{2}-\frac{3i^2}{4}-ij+\frac{\ell j}{2}+\frac{j^2}{4}}\\
    &\le \floor[\Big]{\frac{5n^2}{16}-\frac{n}{8}+\frac{\ell n}{8}-\frac{\ell}{8}-\frac{3\ell^2}{16}+\frac{i\ell}{2}+\frac{i}{2}-\frac{3i^2}{4}+\frac{1}{16}},
\end{align*}
where the last inequality follows from maximizing the concave-up quadratic function of $j$ by substituting $j=(n-1-\ell)/2$ since by \cref{iupperbound} the interval $\frac{n-2}{2}-\frac{3\ell}{4}+\frac{1}{4} \le j \le \frac{n-1-\ell}{2}$ lies to the right of the vertex of the parabola $j=2i-\ell$ for $n\ge6\ell+26$. Therefore 
\begin{align*}
    &e(\T[n-1][3]) + (\ell+1)-e(G) \\
    &>\frac{n^2}{3}-\frac{2n}{3}+\frac{3\ell+4}{3}-1-\left(\frac{5n^2}{16}-\frac{n}{8}+\frac{\ell n}{8}+\frac{1}{16}-\frac{3i^2}{4}+\frac{i}{2}+\frac{i\ell}{2}-\frac{\ell}{8}-\frac{3\ell^2}{16}\right)\\
    &=\frac{1}{48}\left(n^2-26n-6\ell n+36i^2-24i\ell-24i+54\ell+13+9\ell^2\right)\ge 0,
\end{align*}
where the last inequality follows from $n^2-26n-6\ell n= n(n-26-6\ell) \ge 0$ and \cref{eq:54l}.\qedhere
\end{case}
\end{proof}

\begin{thm}\label{thm:k4general}
    Let $\ell \ge -1$ be an integer and $n \geq 6\ell+26$. Let $G$ be an $n$-vertex, $K_{4}$-free graph with degrees $d_1 \le \cdots \le d_n$. If there is an integer $j$ in $1 \le j \le \frac{n-1-\ell}{2}$ such that $d_j \le j+\ell$, then \[ e(G) \leq e(\T[n-1][3]) + (\ell+1)  = \floor[\Big]{\frac{n^2}{3} - \frac{2n}{3} +\frac{3\ell+4}{3}}. \] The set $\Gell$ is nonempty, and equality holds if and only if $G \in \Gell[3][n]$.
\end{thm}

\begin{proof}
    The fact that $\Gell$ is nonempty follows from $n \geq 6\ell+26 \ge (1+1/(r-1))(\ell+1)+1$ for $r=3$ and \cref{def:graphs}. Then this theorem follows from Theorems \ref{thm:k4general1} and \ref{thm:k4general2}.
\end{proof}

\section{Edge Conditions and Extremal Graphs}\label{sec:edge}

In this section we use the Chv\'atal-type degree conditions from \cref{thm:kstable} together with the results of \cref{sec:graphs} and Theorems \ref{thm:kr+1general8}, \ref{thm:kr+1general}, and \ref{thm:k4general} to prove our main result on the maximum number of edges in $n$-vertex, $K_{r+1}$-free, non-Hamiltonian graphs and the extremal graphs attaining these values.

\subsection{Lemmas for exceptional extremal graphs}

The lemmas in this subsection are used to characterize the extremal graphs outside of $\Gell$. Recall that Theorems \ref{thm:kr+1general8} and \ref{thm:k4general} state that for $r=3$ or $r \ge 8$ and sufficiently large $n$, equality holds if and only if $G \in \Gell[r][n][\ell]$, but \cref{thm:kr+1general} leaves open the possibility of extremal graphs in $\Hl$. In this section we determine these exceptional extremal graphs for $4 \le r \le 7$. First we narrow down the values of $n$ where they occur.

\begin{lem}\label{lem:nrange} Let $\ell \ge -1$, $4 \le r \le 7$, and $n \geq \max\{3+ \ell + \frac{4(\ell+2)}{r-3},5+\ell+\frac{r+2\ell+7}{2r-2}\}$ be integers. Let $G$ be an $n$-vertex, $K_{r+1}$-free graph. Suppose that $G \notin \Gell$ 
and \[ e(G) = e(\T[n-1][r]) + (\ell+1)  = \floor[\Big]{\frac{r-1}{2r}n^2 - \frac{r-1}{r}n +\frac{2\ell r + 3r-1}{2r}}.\]
Then 
\begin{itemize}
    \item if $r=4$, then $5\ell+11 \le n \le 13-\ell$,
    \item if $\ell=-1$ and $5 \le r \le 7$, then $n \le 5$,
    \item if $\ell\ge 0$ and $r=5$, then $3\ell+7\le n \le \min\set{17-\ell,10\ell+10}$, and
    \item if $\ell\ge 0$ and $6 \le r \le 7$, then $n \le \min\{4r-3-\ell,6\ell+6\}$.
\end{itemize}
\end{lem}

\begin{proof}
    Suppose $G \notin \Gell$, $e(G) = e(\T[n-1][r]) + (\ell+1)$, and $G$ has degrees $d_1 \le \cdots \le d_n$. By \cref{thm:kr+1general}, $G \in \Hl$ so $G$ satisfies $d_j \le j+\ell$ for $j=(n-1-\ell)/2$. By \cref{Hfamsize}, $n \le 4r-3-\ell$ and so for all $r$ we have $3+ \ell + \frac{4(\ell+2)}{r-3} \le n \le 4r-3-\ell$. Substituting $r=4$ and $r=5$ yields the bounds $5\ell+11 \le n \le 13-\ell$ and $3\ell+7\le n \le 17-\ell$, respectively.     

    For $r \ge 5$ we can further restrict these bounds in order to obtain an upper bound not dependent on $r$. By \cref{Hfamsize}, $G$ consists of an independent set $V(J)$ of $(n-1-\ell)/2$ vertices of degree $(n-1+\ell)/2$ and $H \cong \T[(n+1+\ell)/2][r]$.  
    Notice \begin{align*}
e(G) &= \abs{\partial J}+ e(H)\le \frac{n-1-\ell}{2}\cdot \frac{n-1+\ell}{2} + e(K_{\frac{n+1+\ell}{2}})\\
&= \frac{n^2-2n+1-\ell^2}{4}+\frac{1}{2}\Big( \frac{n+1+\ell}{2}\Big)\Big(\frac{n-1+\ell}{2}\Big)= \frac{3n^2-4n+2\ell n+1-\ell^2}{8}.
\end{align*}
We show that, under certain conditions, we have
\begin{equation}\label{eq:lem}
\frac{3n^2-4n+2\ell n+1-\ell^2}{8} \le \frac{r-1}{2r}n^2-\frac{r-1}{r}n+\frac{2\ell r+3r-1}{2r}-1,
\end{equation}
and therefore
\begin{align*}
e(G) &\le \frac{3n^2-4n+2\ell n+1-\ell^2}{8}\le \frac{r-1}{2r}n^2-\frac{r-1}{r}n+\frac{2\ell r+3r-1}{2r}-1\\
&< \Big\lfloor{\frac{r-1}{2r}n^2-\frac{r-1}{r}n+\frac{2\ell r+3r-1}{2r}}\Big\rfloor= e(\T[n-1][r])+(\ell+1),
\end{align*}
contradicting $e(G) = e(\T[n-1][r])+(\ell+1)$. 

When $\ell=-1$ and $n \ge 6$ for $r \ge 5$, we have \begin{align*}
    \frac{3n^2-6n}{8} \le \frac{2n^2-4n-3}{5} \le \frac{r-1}{2r}n^2-\frac{r-1}{r}n+\frac{2\ell r+3r-1}{2r}-1,
\end{align*} as the last expression is an increasing function of $r$ for every fixed $n \ge 2$, so Inequality (\ref{eq:lem}) holds for $\ell=-1$. We now consider when $\ell \ge 0$.

First, when $n \ge 10\ell+11$ for $r \ge 5$,
\begin{align*}
    10\ell n-5\ell^2-40\ell \le 10\ell(n-1) &\le (n-11)(n-1) = n^2-12n+11\\
    8n+8\ell n -4\ell^2-32\ell-8 &\le n^2-4n-2\ell n+\ell^2+8\ell+3\\
    4+\frac{8n+8\ell n -4\ell^2-32\ell-8}{n^2-4n-2\ell n+\ell^2+8\ell+3} &\le 5 \le r.
\end{align*} Similarly, when $n \ge 6\ell+7$ for $r \ge 6$,
\begin{align*}
    6\ell n-3\ell^2-24\ell \le 6\ell(n-1) &\le (n-7)(n-1) = n^2-8n+7\\
    4n+4\ell n -2\ell^2-16\ell-4 &\le n^2-4n-2\ell n+\ell^2+8\ell+3\\
    4+\frac{8n+8\ell n -4\ell^2-32\ell-8}{n^2-4n-2\ell n+\ell^2+8\ell+3} &\le 6 \le r.
\end{align*} Therefore, when $n \ge 10\ell+11$ for $r \ge 5$, or $n \ge 6\ell+7$ for $r \ge 6$, we have
\begin{align*}
\frac{4n^2-8n+4}{n^2-4n-2\ell n+\ell^2+8\ell+3} &= 4+\frac{8n+8\ell n -4\ell^2-32\ell-8}{n^2-4n-2\ell n+\ell^2+8\ell+3} \le r\\
4n^2-8n+4 &\le rn^2-4rn-2\ell rn+\ell^2 r+8\ell r+3r\\
3rn^2-4rn+2\ell rn+r-\ell^2 r &\le 4rn^2-4n^2-8rn+8n+8\ell r+4r-4 \\
r(3n^2-4n+2\ell n+1-\ell^2) &\le 4(r-1)n^2-8(r-1)n+4(2\ell r+r-1)\\
\frac{3n^2-4n+2\ell n+1-\ell^2}{8} &\le \frac{r-1}{2r}n^2-\frac{r-1}{r}n+\frac{2\ell r+3r-1}{2r}-1,
\end{align*}
so Inequality (\ref{eq:lem}) holds, and we have a contradiction. Therefore, for $\ell=-1$ when $r \ge 5$ we have $n \le 5$. For $\ell \ge 0$, when $r\ge5$ we have $n \le 10\ell+10$, and when $r\ge 6$ we have $n \le 6\ell+6$.
\end{proof}

The following lemma uses \cref{lem:nrange} to simplify our characterization of the exceptional extremal graphs for \cref{theorem:kr+1all}. 

\begin{defn}
    Let $\Hlhat[r][n][0] = \set{K_{6,2,2,1}, K_{4,1,1,1}, K_{5,1,1,1,1}}$, the set of graphs listed in \cref{lem:exceptions}.
\end{defn}

\begin{lem}\label{lem:exceptions}
    If $G \in \Hl[r][n][0]$ for $4 \le r \le 7$, $n \ge r+2$, $n \geq \max\{3 + \frac{8}{r-3},5+\frac{r+7}{2r-2}\}$, and $e(G) = e(\T[n-1][r])+1$, then $(G,r,n) \in \set{(K_{6,2,2,1},4,11), (K_{4,1,1,1},5,7), (K_{5,1,1,1,1},5,9)}$. 
    The converse holds too: if $(G,r,n)$ is in the list above then $G \in \Hl[r][n][0]$, $4 \le r \le 7$, $n \ge r+2$, $n \geq \max\{3+ \frac{8}{r-3},5+\frac{r+7}{2r-2}\}$, and $e(G) = e(\T[n-1][r])+1$.
\end{lem}

\begin{proof}As $G \in \Hl[r][n][0]$, for $j=(n-1)/2$, $G$ satisfies the hypotheses of \cref{thm:kr+1general}, so has the maximum number of edges. By \cref{eq:4} with $j=(n-1)/2$, we have 
        \begin{equation}\label{eq:Hedges}
            e(G) = \Big\lfloor{\frac{r-1}{2r}n^2-\Big(\frac{r-1}{r}n\Big)\frac{n-1}{2}+\frac{3r-1}{2r}\cdot\frac{(n-1)^2}{4}}\Big\rfloor.
        \end{equation}
    Recall from \cref{def:graphs} that $\Hl[r][n][0]$ is empty when $n$ is even, so here $n$ is odd.
    
    First suppose $r=4$.  
        By \cref{lem:nrange}, we have $n \in \set{11,13}$. By \cref{Hfam}, $G \in \set{K_{6,2,2,1}, K_{7,2,2,2}}$. 
        Using \cref{eq:Hedges} to find $e(G)$ and comparing it to $e(\T[n-1][4])+1$, we find that $K_{6,2,2,1} \in \Hl[4][11][0]$ is an extremal graph, but $K_{7,2,2,2} \in \Hl[4][13][0]$ has only 54 edges, while $e(\T[12][4])+1=55$. 
    Now let $r=5$. 
        By \cref{lem:nrange} we have $n \in \set{7,9}$. By the extremality of $G$ and \cref{Hfam}, $G \in \set{K_{4,1,1,1},K_{5,1,1,1,1}}$. By \cref{eq:Hedges} both graphs have $e(\T[n-1][5])+1$ edges so are extremal. Finally, if $r \in \set{6,7}$, then $n\ge r+2\ge 8$, but by \cref{lem:nrange}, $n \le 6\ell+6=6$, a contradiction.
\end{proof}

    \subsection{Main theorem}\label{subsec:omnibus}

The following theorem gives the maximum number of edges in $n$-vertex, $K_{r+1}$-free, non-Hamiltonian graphs 
and characterizes the extremal graphs.

\begin{thm}\label{theorem:kr+1all}
Let $G$ be an $n$-vertex, $K_{r+1}$-free graph where $r\ge 3$.
If $G$ is not Hamiltonian and \[n \ge \begin{cases}26 & \text{if }r=3\\ 11 & \text{if } r=4\\
2 & \text{if } r \ge 5\end{cases},\quad\text{ then }e(G) \leq e(\T[n-1][r]) + 1.\] Equality holds if and only if $G \in \Gell[r][n][0]$ or $G$ is $K_{3,1,1}$, $K_{6,2,2,1}$, $K_{4,1,1,1}$, or $K_{5,1,1,1,1}$, with the exceptional graphs occurring in the cases where $r \ge 5$ and $n =5$, or $(r,n)$ is $(4,11)$, $(5,7)$, or $(5,9)$, respectively. As $\G \in \Gell[r][n][0]$, the bound is tight.
\end{thm}

\begin{proof}
    Equality is achieved if $G \in \Gell[r][n][0]$ by \cref{prop:gell is extremal}. There is an integer $j$ in $1 \le j \le (n-1)/2$ such that $d_j \le j$, using \cref{thm:kstable}.

    When $r=3$, we have assumed that $n \ge 26$. Therefore \cref{thm:k4general} implies that $e(G) \le e(\T[n-1][r]) +1$, that $\Gell[r][n][0]$ is nonempty, and that equality holds if and only if $G \in \Gell[r][n][0]$ for $r=3$.

    If $n-1 \le r$ then the graphs consisting of $K_{n-1}$ and a vertex of degree $1$ are $K_{r+1}$-free and have the maximum number of edges among all $n$-vertex non-Hamiltonian graphs by \cref{thm:ore}. Therefore this graph also is extremal among such graphs that are $K_{r+1}$-free, and $e(G) \le e(K_{n-1})+1 = e(\T[n-1][r])+1$, with equality if and only if $G$ is a $K_{n-1} = \T[n-1][r]$ plus a vertex of degree $1$, with one exceptional graph: $K_{3,1,1}$. Now we assume $n \ge r+2$.
    
    When $r\ge 8$, we have $n \ge 10 \ge 5 + 16/(r-4)$. \cref{thm:kr+1general8} implies that $e(G) \le e(\T[n-1][r])+1$ and that equality holds if and only if $G \in \Gell[r][n][0]$ for all $r \ge 8$. For $4 \le r \le 7$ we note: \[n \ge \begin{cases} 11 & \text{if } r=4\\ 7 & \text{if } 5 \le r \le 7\end{cases} = \max\{3+\frac{8}{r-3}, 7\} \ge \max\{3+ \frac{8}{r-3},5+\frac{r+7}{2r-2}\}.\]
    
    Therefore for $4 \le r \le 7$, \cref{thm:kr+1general} implies that $e(G) \le e(\T[n-1][r])+1$, that if $G \in \Gell$ then equality holds, and that if we have equality then $G \in \Gell[r][n][\ell]$ or $G \in \Hl[r][n][0]$. By \cref{lem:exceptions}, in the latter case we have $G \in \Hlhat[r][n][0]$. By \cref{prop:partitekHam}, the graphs in $\Hlhat[r][n][0]$ are non-Hamiltonian.
\end{proof}

\section{Clique Density Conditions}\label{sec:clique}

In order to prove clique density conditions in Corollaries \ref{cor:cliques} and \ref{cor:posaclique}, we first state preliminary definitions and a theorem of Frohmader determining the maximum number of $t$-cliques in $m$-edge, $K_{r+1}$-free graphs.

A \emph{$t$-clique} is a set of $t$ vertices all pairs of which are adjacent. We write $\k(G)$ for the number of $t$-cliques in a graph $G$. The \emph{colex} (or \emph{colexicographic}) \emph{order} on finite subsets of $\N$, denoted by $<_C$, is defined by $A <_C B$ if and only if $\max(A\symd B) \in B$. The \emph{$r$-partite colex order} is the restriction of the colex order to the pairs of integers $\set{\set{i,j} \in \binom{\N}{2} : i \not\equiv j \mod{r}}$. The \emph{$r$-partite colex Tur\'an graph on $m$ edges}, denoted by $\CT$, has edge set consisting of the first $m$ pairs of natural numbers $i \not\equiv j \mod{r}$ in the $r$-partite colex order. Its vertex set is the union of all of the edges, i.e., the subset of $\N$ needed to support the edges and avoid isolated vertices.

If we define $n$ by $e(\T[n][r])\le m < e(\T[n+1][r])$, then the $r$-partite colex Tur\'an graph on $m$ edges, $\CT$, has the property that $\T[n][r] \subseteq \CT \subset \T[n+1][r]$. The graph $\CT$ consists of the largest Tur\'an graph possible on $m$ edges together with one more vertex whose degree is determined by $m$ and whose neighbors lie in the appropriate partite sets. Radcliffe and Uzzell \cite{RU18} gave the following graph-theoretic form of Frohmader's theorem, which also relies on the rainbow Kruskal-Katona theorem of Frankl, F\"{u}redi, and Kalai \cite{FFK88}.

\begin{thm}[Frohmader \cite{Frohmader}] \label{thm:EdgeZykov}
	If $G$ is a $K_{r+1}$-free graph with at most $m$ edges and $t \ge 2$, then $\k(G) \leq \k\bigl(\CT\bigr)$.
\end{thm}

\cref{thm:EdgeZykov} immediately implies that when an $r$-partite colex Tur\'{a}n graph maximizes the number of edges in a family $\mathcal{A}$ of $n$-vertex, $K_{r+1}$-free graphs, the same colex Tur\'{a}n graph maximizes the number of $t$-cliques for all $t \ge 2$. \cref{thm:generalcolexturan} gives a general form of this observation.

\begin{thm}\label{thm:generalcolexturan}
    Let $\mathcal{A}$ be a set of $n$-vertex, $K_{r+1}$-free graphs. Let $m = \max\set{e(G):G\in \mathcal{A}}$. Let $G \in \mathcal{A}$. For every $t \ge 2$,
    \[
        \k(G) \le \k(\CT),
    \]
    and if $\CT\in\mathcal{A}$ then this upper bound is tight as equality holds if $G \cong \CT$.
\end{thm}

\begin{proof}
	Let $G \in \mathcal{A}$. Then $G$ is $K_{r+1}$-free by the definition of $\mathcal{A}$, and $e(G) \le m$ by the definition of $m$. By \cref{thm:EdgeZykov}, for all $t \ge 2$, we have $\k(G) \le \k(\CT)$. The condition $\CT\in\mathcal{A}$ ensures that this upper bound is achieved.
\end{proof}

To apply \cref{thm:generalcolexturan} to the set $\mathcal{A}$ of $n$-vertex, $K_{r+1}$-free, non-Hamiltonian graphs, we use the following observation.

\begin{obs}\label{obs:colexTuran} For every $r$ and $m$, the $r$-partite colex Tur\'{a}n graph $\CT$ on $m$ edges is $\Gs[r][n][\ell] \in \Gell$, where $n = \abs{V(\CT)}$ and $\ell = \delta(\CT)-1$.
\end{obs}

We now deduce that the $n$-vertex, $K_{r+1}$-free, non-Hamiltonian graphs which maximize the numbers of edges also maximize the numbers of $t$-cliques for all $t \ge 2$.

\begin{cor}\label{cor:cliques}
    Let $G$ be an $n$-vertex, $K_{r+1}$-free, non-Hamiltonian graph where $r \ge 3$. Let $t \ge 2$. If \[n \ge \begin{cases}26 & \text{if }r=3\\ 11 & \text{if } r=4\\ 
2 & \text{if } r \ge 5\end{cases},\quad\text{then } \k(G) \leq \k(\G).\]
Equality holds if $G \cong \G$, which is $\T[n-1][r]$ plus a vertex of degree $1$.
\end{cor}

\begin{proof}
    Apply \cref{thm:generalcolexturan} with $\A$ being the set of $n$-vertex, $K_{r+1}$-free, non-Hamiltonian graphs. The value of $m$ was determined by \cref{theorem:kr+1all}. 
    Note that $\CT\in\mathcal{A}$ by \cref{obs:colexTuran} and \cref{prop:gell is extremal}.
\end{proof}

We similarly maximize $t$-cliques in $K_{r+1}$-free graphs satisfying the P\'{o}sa-type condition.
\begin{cor}\label{cor:posaclique}
Let $\ell \ge -1$, $r \ge 3$, and \[n \ge \begin{cases} 6\ell+26 & \text{if } r=3\\ \max\{3+ \ell + \frac{4(\ell+2)}{r-3},5+\ell+\frac{r+2\ell+7}{2r-2}\} & \text{if } 4 \le r \le 7\\ \ell+5+\frac{4(\ell+4)}{r-4} & \text{if } r \ge 8\end{cases}\] be integers, or $(\ell,r,n)=(1,8,10)$. Let $G$ be an $n$-vertex, $K_{r+1}$-free graph with degrees $d_1 \le \cdots \le d_n$. If there is an integer $j$ in $1 \le j \le (n-1-\ell)/2$ such that $d_j \le j+\ell$, then, for all $t \ge 2$, $\k(G) \le \k(\Gs)$. Equality holds if $G \cong \Gs$.
\end{cor}

\begin{proof}
    Apply \cref{thm:generalcolexturan} with $\A$ being the set of $n$-vertex, $K_{r+1}$-free graphs with degrees $d_1 \le \cdots \le d_n$ having some integer $j$ in $1 \le j \le (n-1-\ell)/2$ such that $d_j \le j+\ell$. The value of $m$ was determined by \cref{thm:degcondsummary}. We have $\CT\in\mathcal{A}$ by \cref{obs:colexTuran} and \cref{prop:gell is extremal}.
\end{proof}

\section{Chorded Pancyclicity}\label{sec:chordedpancyclicity}

A graph on $n\ge 4$ vertices is \emph{chorded pancyclic} if for every $4 \le k\le n$ the graph contains a cycle of length $k$ that has a chord. Sufficiently dense Hamiltonian graphs are chorded pancyclic.

\begin{thm}[Chen, Gould, Gu and Saito \cite{CGGS18}]\label{thm:chordedpancyclic} Let $G$ be a Hamiltonian graph of order $n\ge 4$. If $G$ has at least $n^2/4$ edges, then $G$ is chorded pancyclic, unless $n$ is even and $G = K_{n/2,n/2}$ or $G = K_3\square K_2$, the cartesian product of $K_3$ and $K_2$.\end{thm}

We use \cref{thm:chordedpancyclic} to extend \cref{theorem:kr+1all} to the following analogous statement.

\begin{cor}\label{cor:chord}
Let $G$ be an $n$-vertex, $K_{r+1}$-free graph where $r \ge 3$. If $G$ is not chorded pancyclic and 
\[n \ge \begin{cases}26 & \text{if } r=3\\ 11 & \text{if } r=4\\ 4 & \text{if } r \ge 5\end{cases},\quad\text{then }e(G) \le e(\G) = e(\T[n-1][r])+1
.\] Equality holds if $G \cong \G$, which is a $\T[n-1][r]$ plus a vertex of degree $1$.
\end{cor}

\begin{proof}
Suppose that 
$e(G) > e(\T[n-1][r])+1$, and we will prove that $G$ is chorded pancyclic. 
By \cref{theorem:kr+1all}, $G$ is Hamiltonian. If $4 \le n \le r$, we have 
    \[
        e(\T[n-1][r])+1 = e(K_{n-1})+1 = \frac{(n-1)(n-2)}{2}+1 \ge \frac{n^2}{4}.
    \]
    Otherwise, we have $n \ge r+1$, so $n \ge 6$, and 
\begin{align*}
    n &\ge \frac{9}{2}-\frac{1}{n} \ge \frac{9}{2}-\frac{1}{r}=\frac{4r}{r-2}\left(\frac{r-1}{r}+\frac{r^2-12r+4}{8r^2}\right),\quad\text{so,}\\
    \left(\frac{r-1}{2r}-\frac{1}{4}\right)n^2 = \frac{r-2}{4r}n^2 &\ge \frac{r-1}{r}n+\frac{r^2-12r+4}{8r^2}n \ge \frac{r-1}{r}n+\frac{r^2-12r+4}{8r}
\end{align*} and by \cref{prop:turanedgecount} and rearranging terms of the previous line, we get
\[
    e(\T[n-1][r])+1 \ge \frac{r-1}{2r}(n-1)^2-\frac{r}{8}+1 = \frac{r-1}{2r}n^2 -\frac{r-1}{r}n-\frac{r^2-12r+4}{8r}
    \ge \frac{n^2}{4}.
\]
In both cases we have $e(G) > e(\T[n-1][r])+1 \ge n^2/4$, so $G$ satisfies the hypotheses of \cref{thm:chordedpancyclic} and is not $K_{n/2,n/2}$.

When $n=6$, 
we have $r \ge 5$ and 
$e(G) > e(\T[5][r])+1 = e(K_5)+1 = 11>9=e(K_3 \square K_2)$. 
Therefore $G\neq K_{n/2,n/2}$ and $G \ne K_3 \square K_2$, so by \cref{thm:chordedpancyclic} $G$ is chorded pancyclic.

If $G$ is a $\T[n-1][r]$ plus a vertex of degree $1$, then $G$ is $K_{r+1}$-free, not chorded pancyclic, and has the required number of edges.
\end{proof}

\begin{cor}\label{cor:chordclique}
Let $G$ be an $n$-vertex, $K_{r+1}$-free graph where $r \ge 3$. Let $t \ge 2$. If $G$ is not chorded pancyclic and 
\[n \ge \begin{cases}26 & \text{if } r=3\\ 11 & \text{if } r=4\\ 4 & \text{if } r \ge 5\end{cases},\quad\text{then }k^t(G) \le \k(\G)
.\] Equality holds if $G$ is $\G$, which is a $\T[n-1][r]$ plus a vertex of degree $1$.
\end{cor}
\begin{proof}
    Apply \cref{thm:generalcolexturan} with $\A$ being the set of $n$-vertex, $K_{r+1}$-free graphs that are not chorded pancyclic, with $m$ determined by \cref{cor:chord}. Again $\CT\in\mathcal{A}$ by \cref{obs:colexTuran} and \cref{prop:gell is extremal}, as every chorded pancyclic graph is Hamiltonian.
\end{proof}

\section{Edge Density Condition for $r=2$}\label{sec:r2}

The $r=2$ ($K_3$-free) case is quite different. Here the extremal graphs are complete bipartite graphs that are almost balanced.

\begin{thm}\label{thm:r2Ham} 
    Let $n \ge 3$. If $G$ is an $n$-vertex, $K_3$-free graph that is not Hamiltonian, then \[e(G) \le e(K_{\ceil{n/2}-1,\floor{n/2}+1}),\]
    and this bound is tight as $K_{\ceil{n/2}-1,\floor{n/2}+1}$ is an $n$-vertex, $K_3$-free graph that is not Hamiltonian.
\end{thm}
\begin{proof}
    Let $G$ be such a graph. By Tur\'an's theorem (\cref{thm:turan}), every $n$-vertex, $K_3$-free graph $G$ has $e(G) \le \floor{n^2/4}$, with equality if and only if $G \cong K_{\floor{n/2},\ceil{n/2}}$. If $n$ is odd then $K_{\floor{n/2},\ceil{n/2}}$ is not Hamiltonian so gives the maximum number of edges. If $n \ge 4$ is even then $K_{\floor{n/2},\ceil{n/2}}$ is Hamiltonian, so $e(G)=n^2/4$ is not possible. The graph $K_{n/2-1,n/2+1}$ is not Hamiltonian, so $e(K_{n/2-1,n/2+1}) = n^2/4-1$ gives the maximum number of edges.
\end{proof}

\section{Open Problems}\label{sec:open}

Letting $\ex_{\Ham}(n,F)$ denote the maximum number of edges in a non-Hamiltonian, $n$-vertex, $F$-free graph, and letting $\ex_{\Ham}(n,H,F)$ denote the maximum number of copies of a graph $H$ in a non-Hamiltonian, $n$-vertex, $F$-free graph, suggests an abundance of open problems, as we have addressed only the case $F=K_{r+1}$ and $H = K_t$. One interesting variant is as follows, fixing the number of edges instead of the number of vertices, as in Frohmader's theorem (\cref{thm:EdgeZykov}).

\begin{qu}What is the maximum number of $t$-cliques among the $m$-edge, $K_{r+1}$-free, non-Hamiltonian graphs?\end{qu}

Kang and Pikhurko \cite{KangPikhurko} determined the maximum number of edges in $n$-vertex, $K_{r+1}$-free graphs which are not $r$-partite. By \cref{theorem:kr+1all} the extremal graphs all are Hamiltonian.

\begin{qu}What is the maximum number of edges in $n$-vertex, non-Hamiltonian, $K_{r+1}$-free graphs which are not $r$-partite?\end{qu}

As discussed in the introduction (see \cref{fig:tofl}), the extremal graphs for the problem of maximizing the number of $t$-cliques in $n$-vertex, $K_{r+1}$-free, non-Hamiltonian graphs can be viewed as the result of attaching a pendant edge to the extremal graph for the problem of maximizing the number of $t$-cliques in $(n-1)$-vertex, $K_{r+1}$-free graphs. Similarly the extremal graph among $n$-vertex non-Hamiltonian graphs is obtained by attaching a pendant edge to the extremal graph among $(n-1)$-vertex graphs. One might wonder how general this phenomenon is. The $r=2$ case in \cref{sec:r2} is an example where it does not hold. Recent work of F\"{u}redi, Kostochka, and Luo \cite[Corollary 5]{FKL18} shows that for every graph $H$, the Ore graph in \cref{fig:tofl} maximizes the number of copies of $H$ in $n$-vertex, non-Hamiltonian graphs for sufficiently large $n$. Similarly Morrison, Nir, Norin, Rz\textogonekcentered{a}$\dot{z}$ewski, and Wesolek \cite{MNNRW23} showed that for every $H$ the number of copies of $H$ in $n$-vertex $K_{r+1}$-free graphs is maximized by the Tur\'an graph for sufficiently large $n$ and $r$. It seems natural to conjecture that for every $H$ the number of copies of $H$ in an $n$-vertex, $K_{r+1}$-free, non-Hamiltonian graph is maximized, for sufficiently large $n$ and $r$, by $\G$. By \cref{cor:cliques} this conjecture holds for $H=K_t$.
\begin{qu}
    For which graphs $H$ is this conjecture true?
\end{qu}

\section*{Acknowledgments}

The authors thank Jamie Radcliffe for comments which improved the presentation of the paper, and V\'{a}clav Chv\'{a}tal, Linda Lesniak, Zhanar Berikkyzy, Kirsten Hogenson, and Jessica McDonald for helpful discussions. The first author is supported in part by the NSF Grant DMS-2402312. The second author is supported in part by Simons Foundation Grant MP-TSM-00002688.

\bibliographystyle{plainurl}
\bibliography{references}
\end{document}